\documentclass[11pt,twoside,reqno]{amsart}

\usepackage{amssymb}

\def\res{\text{\rm Res}}
\def\const{\text{\rm const}}

\def\ti{\tilde}

\def\dist{\text{\rm dist}}
\def\supp{\text{\rm supp}\,}
\def\to{\rightarrow}

\def\bs{\bigskip}

\def\ms{\medskip}
\def\no{\noindent}

\def\R{{\mathbb R}}

\def\Z{{\mathbb{Z}}}
\def\C{{\mathbb{C}}}
\def\N{{\mathbb{N}}}

\def\MM{{\mathcal M}}

\def\SS{{\mathcal S}}
\def\HH{{\mathcal H}}

\def\EE{{\mathcal E}}

\def\e{\varepsilon}

\def\L{\Lambda}
\def\l{\lambda}

\def\lan{\lambda_n}
\def\gan{\gamma_n}

\def\Sch{Schr\"odinger }

\theoremstyle{plain}
\newtheorem{lemma}{Lemma}
\newtheorem{theorem}{Theorem}
\newtheorem{corollary}{Corollary}
\newtheorem{proposition}{Proposition}

\newtheorem{remark}{Remark}

\newtheorem{example}{Example}

\numberwithin{equation}{section}

\author{N.~Makarov}
\thanks{The first author is supported by
N.S.F. Grant DMS-1500821}
\address{California Institute of Technology\\
Department of Mathematics\\
Pasadena, CA 91125, USA}
\email{makarov@its.caltech.edu}

\author{A.~Poltoratski}
\address{Texas A\&M University
\\ Department of Mathematics\\
College Station, TX 77843, USA }
\email{alexeip@math.tamu.edu}
\thanks{The second author is supported by
N.S.F. Grant DMS-1665264}

\title{Two-spectra theorem with uncertainty}

\begin{document}

\begin{abstract} The goal of this paper is to combine ideas from the theory of mixed spectral
problems for differential operators with new results in the area of the Uncertainty Principle in Harmonic Analysis (UP).
Using  recent solutions of Gap and Type Problems of UP we prove a version of Borg's two-spectra theorem for \Sch operators, allowing uncertainty in the placement
of the eigenvalues. We give a formula for the exact 'size of uncertainty', calculated from the lengths of the intervals where the eigenvalues may occur. Among other applications,
we  describe pairs of indeterminate operators in the three-interval case of the mixed spectral problem.
At the end of the paper we discuss further questions and open problems.

\end{abstract}

\maketitle

\section{Introduction}

\ms\no This paper studies connections between classical Harmonic Analysis and spectral problems for differential operators.
A newly developed approach in the area of the Uncertainty Principle in Harmonic Analysis (UP), based on the use
of Toeplitz operators, has brought new progress to several long-standing problems. The Toeplitz approach, pioneered
 by Hruschev, Nikolski and Pavlov in \cite{Pa, HNP}, was further developed and applied to a broader class of problems
in \cite{MIF1, MIF2}. Recent applications of the Toeplitz approach in \cite{Polya, Gap, Type, Poly, Det, CBMS} involved a variety of
problems of UP and adjacent fields. Two of such problems, the so-called Gap and Type Problems, will be revisited in this note.

\ms\no While most of the remaining open problems of UP are extremely hard,  progress in any of them usually invites numerous applications in
related fields. One of such related fields is spectral theory for differential equations and Krein's canonical systems.
The so-called Weyl-Titchmarsh transform, which can be viewed as a generalization of the Fourier transform for a broad class of
linear differential operators, provides clear analogies between spectral theory and Fourier analysis. However, despite
being clearly visible at the intuitive level, such analogies
can take time and significant effort to formulate in rigorous mathematical language. One of the recent examples of such relations is
the result on equivalence between the so-called Beurling-Malliavin problem on completeness of complex exponentials, solved in 1960s,
and mixed spectral problems for \Sch operators, found by M. Horvath only in 2004 (\cite{Horvath}; see also \cite{MIF1}).


\ms\no As was mentioned above, the goal of this paper is to discuss the connections between  the Gap  and Type Problems and the so-called mixed spectral problems for second order differential operators.
Surprisingly, such connections do not seem to be  reflected in the literature. While most of our results on spectral problems could be formulated
in more general classes of differential operators, such as Krein's canonical systems, we choose to test our methods on one the most studied models,
the \Sch (Sturm-Liouville) operator on an interval.

\ms\no We consider the equation
\begin{equation}-u''+qu=z^2u \end{equation}
on the interval $[0,\pi]$. To simplify the statements and formulas, we will always assume that the potential
function $q$ belongs to $L^2([0,\pi])$, although most of our results admit $L^p$-analogs with $p\geq 1$.

\ms\no Classical
spectral problems are divided into two main groups: direct and inverse problems. In direct problems one needs to find the spectra or the spectral measure
of the operator from its potential $q$, while in inverse problems one aims to recover $q$ from the spectral information.
A famous result by Marchenko \cite{M1, M2} says that $q$ can be uniquely recovered from the spectral measure. Another classical result is Borg's two-spectra theorem \cite{Borg},
which says that $q$ can be recovered from two spectra, corresponding to different pairs of boundary conditions, see Section \ref{functions} for further discussion.

\ms\no A relatively new type of spectral problems, the so-called mixed spectral problem, asks to recover the operator from partial information on the potential and the spectrum.
The well-known theorem by Hochstadt and Lieberman from 1978 \cite{HL} says that knowing the potential on one-half of the interval $(0,\pi/2)$ and knowing
one of the spectra allows one to recover the operator uniquely. The result is precise in the sense that the knowledge of the spectrum minus one point, or of the potential on $(0,\pi/2-\e)$, is insufficient.

\ms\no Further results in the same direction obtained by   del Rio, Gesztesy, Horvath, Simon and others \cite{S, SG1, DGS, Horvath, MIF1} allow one to combine various parts of the spectral and direct information to recover the operator. Together these results bring about the following
quantitative interpretation of the spectral data.

\ms\no Suppose one is given a certain amount of spectral information on an operator. In general, such information can be
very diverse, from knowing a part of the spectrum, to knowing parts of several spectra, e. g. \cite{DGS, MIF1},
or some of the pointmasses of one or several of the spectral measures. The general question one may ask in such a situation is: What part of the full spectral information is missing?
In some  cases the answer is ready: for instance, as follows from Borg's theorem, one spectrum gives exactly one
half of the full information, since to recover the operator one needs two spectra. Results on the  mixed spectral problems suggest that knowing a subsequence of the spectra of relative density, say, one-half gives one-forth
of the full information, since it needs to be complemented with the knowledge of the potential on three-fourths of the interval. Knowing
the spectra of two restrictions of the original operator on $(0,a)$ and $(a,\pi)$ gives one half of the information if the spectra are disjoint, because to recover the operator one needs to complement it with one spectrum of the whole operator \cite{SG2},  etc.

\ms\no The ideas of mixed spectral problems allow us to formulate this rather vague question in precise mathematical terms.
We may say that the given spectral information lacks  a part of size $a,\ 0\leq a\leq 1,$ if together with the
potential on any $(0,a\pi+\e),\ \e>0$ this information allows one to determine the operator uniquely, while knowing the potential on $(0,a\pi-\e)$ is insufficient. It becomes an interesting problem to determine $a$ for various sets of (incomplete) spectral information.

\ms\no In the present paper we answer this question for the 'uncertainty' version of Borg's problem. Suppose that instead
of knowing two spectra of a \Sch operator we are given a sequence of intervals on the real line where the eigenvalues from
the two spectra must occur. Clearly, this spectral information is incomplete, but what part of the full information does it constitute?
This question can be considered within the context of Uncertainty Quantification, which is an active research area in natural sciences, engineering and numerical analysis.

\ms\no
In Section \ref{2int} we give a formula to find $a$ in this problem, which of course depends on the lengths of the intervals.
As we will see, this question turns out to be closely related to the Gap and Type  Problems of UP. Our formula for the size of uncertainty in the
 two-spectra problem
is equivalent to a case of the Gap Problem which can only be solved with  recent results of \cite{Polya, Gap, Type, CBMS}.

\ms\no Let $\mu$ be a non-zero finite complex measure on the real line.
By $\hat \mu$ we denote its Fourier transform
$$\hat\mu(z)=\int\exp(-izt)d\mu(t).$$
In many applications a special role is played by measures with a spectral gap, i.e., those non-zero measures for which
$\hat\mu=0$ on $(-a,a)$ for some $a>0$. Such measures are associated, for instance, with the exponential version of the moment problem,
determinacy properties in the theory of random processes and the study of high-pass signals in electrical engineering.
In its most general formulation, the Gap Problem asks for necessary and sufficient conditions for a measure to have a spectral gap of a
given size.

\ms\no The original statement of UP, as formulated by Norbert Wiener, says that a measure (function, distribution) and its Fourier transform cannot be simultaneously small.
The Gap Problem traditionally belongs to the area of UP.
 Here the smallness of $\hat\mu$ is understood in the sense of the size of the gap. The statement that one hopes to obtain is that if the support of $\hat\mu$ has a large gap, then
the support of $\mu$ cannot be too "rare", its weight cannot decay too fast, etc.
For more on the history of the Gap Problem see
\cite{  Beurling1,  Koosis, Levinson,    CBMS}.


\ms\no The area of mixed spectral problems has direct connections to the Gap Problem. The first evidence of such a connection can be seen from
one of the main results of \cite{S}, which says that the difference of the Weyl functions of two operators whose potentials coincide on $(0,c)$
must decay fast at infinity, see Section \ref{FSGap}. Experts in the Gap Problem could recognize such a decay condition
as an estimate for the Cauchy transform of a measure with a spectral gap, after a square root substitution. In Section \ref{FSGap} we will see
that the potentials of two operators coincide on $(0,a\pi)$ if and only if the difference of the two spectral measures has the  gap
$(-2a,2a)$ in its Fourier spectrum.

\ms\no Another result of this paper is a 'parametrization' of counterexamples (indeterminate operators) in the so-called three-interval case of the mixed spectral problem.
After the theorem by Hochstadt and Lieberman mentioned above and further results by del Rio, Gesztesy, Simon and Horvath the case of
two intervals, i.e., the case when the potential is known on $(0,a),\ 0<a<\pi$ and unknown on $(a,\pi)$, is relatively well-studied. The logical problem is to try to replace
$(0,a)$ with other subsets of $(0,\pi)$, starting with a natural next step of two intervals $(0,a)\cup (b,\pi),\ 0<a<b<\pi$. However, such attempts immediately
meet the following elegant counterexample contained in \cite{SG2}. Let the potential $q$ of a \Sch operator $L$ satisfy $q(x)=q(\pi-x)$ for all $x\in(0,\frac \pi 2-\e)$. Let $\ti L$ be the operator with the reflected potential $\ti q(x)=q(\pi-x)$ for all $x$. It is well known that
then the spectra of the two operators, with any pair of symmetric boundary conditions at the endpoints, will coincide. Nonetheless, the operators are not identical, unless $q(x)=q(\pi-x)$ for $x\in(\frac \pi 2-\e, \frac \pi 2)$.
Thus, having almost complete direct information (knowing the potential on the two intervals of total length $\pi-2\e$) and a spectrum is insufficient to recover the operator.

\ms\no This counterexample shows that the two-interval problem is substantially different from the three-interval case. In view of Horvath's result, see Sections \ref{sHorvath}, \ref{P3}, it seems that these differences are a reflection of the differences between problems of completeness of exponential
functions in $L^2$ over one interval (the classical case) and two or more intervals (the case of band spectrum). Even though the classical case was
solved via the famous Beurling-Malliavin theorem in 1960s, the latter case is still widely open, see for instance \cite{Olevsky}. We discuss
these connections and formulate further questions in Sections \ref{BM} and \ref{P3}.

\ms\no Until now the counterexample for the three-interval problem mentioned above was the only one existing in the literature, prompting some of
the experts to ask if it was, in some sense, the only counterexample of that kind. In Section \ref{s3int} we formulate a result describing all possible counterexamples
for the three-interval problem in terms of the spectral measures of the operators. It follows from our 'parametrization' that the pairs of operators with
'almost symmetric' potentials provide only a  submanifold of all pairs of indeterminate operators.
All other such pairs, however, have to be close to symmetric in terms of the asymptotics of the spectral measure, see Section \ref{s3int}.
In the opposite direction, we show that
there exists a large class of operators which are uniquely determined by their potentials on two intervals and a proper part of the spectrum.
This raises the natural question of describing the operators with such uniqueness properties discussed in Section \ref{P2}.

\ms\no The paper is organized as follows.

\begin{itemize}

\item In Sections \ref{spectra} and \ref{functions} we define the spectra, spectral measures and Weyl functions
for \Sch operators.

\item In Section \ref{even} we discuss special properties of an operator with even potential and point out
a connection between its spectral measure and the classical problem on completeness of polynomials.

\item In Section \ref{HB} we give definitions and discuss basic properties of Hermite-Biehler functions and de Branges spaces.
In Section \ref{HBS} we discuss those spaces in relation to \Sch operators.

\item In Section \ref{HBSthm} we give a description of Hermite-Biehler functions corresponding to \Sch operators with square-summable potentials.

\item In Section \ref{Kshift} we define the Krein spectral shift function for a \Sch equation and prove some of its properties to be used in the main proofs.

\item In Section \ref{BM} we give definitions of Beurling-Malliavin (BM) density and formulate the famous BM Theorem, which solves the
classical problem on completeness of families of exponential functions

\item In Section \ref{Gap} we give an overview of some of the recent results on the Gap and Type Problems to be used in Section \ref{2int}.

\item In Section \ref{sHorvath} we formulate and give a short proof to a theorem by Horvath on the connection between mixed spectral problems for \Sch operators and the BM problem on completeness of exponential systems.

\item In Section \ref{FSGap} we establish the connection between the Gap and Type Problems and mixed spectral problems for \Sch operators.

\item In Section \ref{2int} we formulate and prove a version of the two-spectra theorem with uncertainty mentioned in the introduction.

\item In Section \ref{s3int} we parametrize all examples of non-uniqueness in the three-interval version of the mixed spectral problem. In Section
\ref{3intU} we show the existence of operators uniquely determined by the 3-interval mixed data.

\item In Sections \ref{P1}-\ref{P3} we discuss further examples and open problems

\end{itemize}


\newpage

\section{Holomorphic functions in spectral problems}

\bs\subsection{The \Sch operator and its spectra}\label{spectra}

\ms\no As was mentioned in the introduction, we consider the \Sch equation
\begin{equation}Lu=-u''+qu=z^2u \label{e00}\end{equation}
on the interval $[0,\pi]$. To avoid unnecessary technicalities we will make the following assumptions.

\ms\no Even though some of our results can be extended to the case $q\in L^p$, we mostly restrict ourselves to square summable potentials,  $q\in L^2$,
which makes some of the statements considerably shorter.
 For the same reason
we will consider only Dirichlet ($u=0$) or Neumann ($u'=0$) boundary conditions at the endpoints.

\ms\no The operator
$$L:u\mapsto -u''+qu$$
is a self-adjoint operator in $L^2([0,\pi])$ whose domain is the set of functions $u$ with absolutely continuous derivatives which satisfy the
boundary conditions and the condition $-u''+qu\in L^2([0,\pi])$. In the case $q\in L^2([0,\pi])$ the domain consists of functions from the Sobolev space $W^{2,2}$ satisfying the boundary conditions.
The spectrum of the operator
with $q\in L^1$ and any self-adjoint boundary conditions at the endpoints is bounded from below.
We will  assume the spectrum $\Sigma_{ND}$ of $L$ with Dirichlet-Neumann boundary conditions to be positive. This assumption is not overly restrictive as any operator can be made positive by adding a positive constant to $q$: the spectrum then shifts to the right by the same constant.

\ms\no We will denote the set of  operators satisfying the above conditions by $\SS^2$. Occasionally we will use the notation $\SS^1$ for the operators with summable potentials.

\ms\no As it is often done in this area, to facilitate the application of standard tools of Fourier analysis we will apply the square root transform
to the spectra and analytic functions associated with our operators. In particular,
we will denote by $\sigma_{DD}$ and $\sigma_{ND}$ the spectra of the operator $L$ after the square root transform:
$$\sigma_{DD}=\{\lan|\lan^2\in\Sigma_{DD}\}\cup\{0\}, \ \ \ \sigma_{ND}=\{\lan|\lan^2\in\Sigma_{ND}\}.$$
Note that under our positivity assumptions  both spectra are real.

\ms\no We say that a sequence of complex numbers is discrete if it does not have finite accumulation points. Both spectra $\sigma_{DD}$ and $\sigma_{ND}$
are real discrete sequences.  Moreover, it is well known that
 \begin{equation}\sigma_{DD}=\{\lan\}_{n\in\Z},\ \l_0=0,\ \l_{\pm n}=\pm\pi\left(n+\frac Cn+\frac {a_{n}}n\right),\ n\in\N,\label{e000a}\end{equation}
  and
  \begin{equation}\sigma_{ND}=\{\eta_n\}_{n\in\pm\N},\ \eta_{\pm n}=\pm\pi \left(n -\frac 12+\frac Cn+\frac {b_n}n\right),\ n\in\N ,\label{e000}\end{equation}
 where $C$ is a real constant and $\{a_n\},\{b_n\}\in l^2$. Conversely, two interlacing sequences are equal to the spectra $\sigma_{DD},\ \sigma_{ND}$
 of a \Sch operator on $[0,\pi]$ with an $L^2$-potential if and only if the sequences satisfy the above asymptotics with some $a_n,b_n\in l^2$.
  These asymptotics follow from more general formulas by Marchenko \cite{M2}, see also \cite{Tru} or \cite{BBP}.


 \bs\bs\subsection{ Analytic integrals of spectral measures}\label{functions}

 \ms\no We denote by $\Pi$ the Poisson measure on the real line, $d\Pi(x)=\frac{dx}{1+x^2}$, and
by $ L^1_\Pi$ the space of Poisson-summable functions on $\R$.
If $u\in L^1_\Pi$  we define its Herglotz integral as
$$H u(z)=\frac1{\pi } \int\left[\frac1{t-z}-\frac t{1+t^2}\right]u(t)dt.$$
If $\mu$ is a Poisson finite measure on $\R$,
i.e., $$\int\frac{d|\mu|(x)}{1+x^2}<\infty,$$
 then
$$H \mu(z)=\frac1{\pi } \int\left[\frac1{t-z}-\frac t{1+t^2}\right]d\mu(t).$$

 \ms\no Denote by $u_{z}(t)$ the solution of \eqref{e00} with boundary conditions $u(0)=0,\ u'(0)=1$. The Weyl function $m_+$ is defined as
 $$m_+=-\frac{u_z'(\pi)}{zu_z(\pi)}.$$
 It is well-known that $m_+$ is a Herglotz integral of a positive measure supported on $\sigma_{DD}=\{\lan\}$:

 $$ m_+(z)=H\mu_+(z),$$
 $$\mu_+=\alpha_0\delta_0+\sum_{n\in\N} \alpha_n(\delta_{\lambda_{ n}}+\delta_{\l_{-n}}),$$
 \begin{equation}\alpha_{ n}=1+\frac {c_n}{n+1},\ c_n\in l^2.\label{asymp1}\end{equation}

\ms\no  Similarly, if $v_{z}(t)$ is the solution of \eqref{e00} with boundary conditions $v(\pi)=0,\ v'(\pi)=1$, one defines
 $$m_-=\frac{v_z'(0)}{zv_z(0)}.$$ Then $\mu_-$ is defined via
 $$m_-(z)=H\mu_-(z)$$
 $$\mu_-=\beta_0\delta_0+\sum_{n\in\N} \beta_n(\delta_{\lambda_{ n}}+\delta_{\l_{-n}}),$$
 \begin{equation}\beta_{ n}=1+\frac {d_n}{n+1},\ d_n\in l^2.\label{asymp2}\end{equation}

 \ms\no The asymptotics of the pointmasses $\alpha_n, \beta_n$ can be deduced from the spectral asymptotics of $\sigma_{DD}$
 and $\sigma_{ND}$ discussed in the last section. Moreover, together the asymptotics for $\lan\in \sigma_{DD}$ ($\eta_n\in \sigma_{ND}$)
 and for $\alpha_n$ ($\beta_n$) give an if and only if condition for a positive measure to be a spectral measure $\mu_+$ ($\mu_-$)
 of a \Sch operator from $\SS^2$.
The famous  theorem by Marchenko says that a regular ($q\in L^1$) \Sch operator can be uniquely recovered from its spectral
 measure $\mu_+$ (or $\mu_-$). Another classical result is the following two-spectra theorem by Borg. Throughout the
 paper, we will say that a \Sch operator is uniquely determined by its (mixed) spectral data, if any other
 operator from the same class with the same data must have the same potential $q$ a. e. on $[0,\pi]$.

 \begin{theorem}[Borg, \cite{Borg}]\label{Borg} A \Sch operator $L\in \SS^1$ is uniquely determined by its spectra $\sigma_{DD}$ and $\sigma_{ND}$.
 No proper subset of $\sigma_{DD}\cup\sigma_{ND}$ has the same property.
 \end{theorem}

\ms\no As was mentioned in the introduction, we will present an 'uncerainty' version of this theorem in Section \ref{2int}.




 \ms\no For a subinterval $(a,b)$ of $(0,\pi)$ we denote by $M=M_{(a,b)}$ the transfer matrix, the  matrix valued entire function defined as
 $$M(z)=\left(\begin{array}{cc}
   u_z(b) & v_z(b) \\
   u'_z(b) & v'_z(b)
   \end{array} \right)$$
   where $u_z$ and $v_z$ are solutions for the restriction of the equation \eqref{e00} on the interval
   $(a,b)$ with Dirichlet and Neumann   initial conditions at $a$ respectively. Note that by the Wronskian
   identity, $det M\equiv 1$ and $M(w_z(a),w'_z(a))^T=(w_z(b),w'_z(b))^T$ for any solution $w_z$ of \eqref{e00}.

 \bs\bs\subsection{The even operator}\label{even}

\ms\no We will call an operator $L$ even (with respect to the middle of the interval), if its potential satisfies $q(x)=q(\pi-x)$ for all $x\in(0,\pi/2)$. Such an operator is uniquely
defined by one of its spectra, say $\sigma_{DD}$, see for instance \cite{Tru}. If $\L=\{\lan\}$ is a sequence satisfying \eqref{e000a}, we denote by $\Gamma(\L)=\{\gamma_n\}$ the sequence of pointmasses
of the spectral measure $\mu_+=\sum \gamma_n\delta_{\lan}$ corresponding to the unique even operator with  $\sigma_{DD}=\{\lan\}$. Notice that in the even case $\mu_+=\mu_-$.

\begin{lemma}\label{leven} For any \Sch operator $L\in \SS^1$ the pointmasses of $\mu_\pm$ satisfy
$$\alpha_n\beta_n=\gamma_n^2.$$

\end{lemma}

\begin{proof} Let $u_z,v_z$ be solutions defined as above for $L$. Notice that for the even operator $L^*$ such solutions coincide up to a constant multiple and denote
by $w_z$ one of these solutions. The functions $zu_z(\pi), zv_z(0)$ and $zw_z(\pi)$ are entire functions of $z$ of exponential type $\pi$  which are real on the real line and have zeros at $\lan$. Since such an entire function is unique up to a real constant multiple, $zu_z(\pi), zv_z(0)$ and $zw_z(\pi)$ are equal up to
a constant multiple. Since all three functions must be equivalent to $\sin z/z$ as $z=iy, y\to\infty$ (see for instance \cite{Tru}, page 13, Theorem 3), they must be the same. Denote this entire function by $F(z)$.

\ms\no Let $\lambda\in\sigma_{DD}$. Let $u_\lambda'(\pi)=C$ and $v_\lambda'(0)=D$. Notice that uniqueness of Dirichlet-Dirichlet solution implies $C=1/D$.

\ms\no Since
$$m_+(z)m_-(z)=\frac{u_\lambda'(\pi)v_\lambda'(0)}{[F'(\lambda)]^2}(z-\lambda)^2+O(1),\ \textrm{as }z\to\lambda,$$
we conclude that
$$\alpha_n\beta_n=\frac 1{[F'(\lambda_n)]^2}.$$
It is left to notice that since the right-hand side does not depend on $L$, for the even operator we must have
$$\alpha^*_n\beta^*_n=\gamma^2_n=\frac 1{[F'(\lambda_n)]^2}.$$

\end{proof}

\begin{remark}\normalfont The constants $\gamma_n$, the pointmasses of the spectral measure of the even operator, appear under different names in other problems
of analysis. As an example, let us point out the following connection with the classical problem of completeness of polynomials in weighted spaces.

\ms\no Throughout the paper, we assume that all discrete real sequences are enumerated in natural increasing order.

\ms\no   We  say that a sequence $\L=\{\lan\}$ has (two-sided) upper density $d$ if
$$\limsup_{A\to\infty}\frac{\#[\L\cap (-A,A)]}{2A}=d.$$
If $d=0$ we say that the sequence has zero density.
  A discrete sequence $\L=\{\lan\}$ is called \textit{balanced} if the limit
\begin{equation}
\lim_{N\to\infty}\sum_{|n|<N}\frac {\l_n}{1+\lan^2}
\label{balance}
\end{equation}
exists.

\ms\no  Let $\L=\{\lan\}$ be a balanced sequence of finite upper density. For each $n, \ \lan\in\L,$ put
$$p_n= \frac 12\left[\log(1+\l_n^2)+\sum_{n\neq k, \ \l_k\in\L}\log\frac{1+\l_k^2}{(\l_k-\lan)^2}\right],$$
where the sum is understood in the sense of principle value, i.e. as
$$\lim_{N\to\infty}\sum_{0<|n-k|<N}\log\frac{1+\l_k^2}{(\l_k-\lan)^2}.$$
We will call the sequence of such numbers $P(\L)=\{p_n\}$ the \textit{characteristic sequence} of $\L$.

\ms\no Here is a sample of a statement on completeness of polynomials in terms of characteristic sequences.

\begin{theorem}\cite{Poly, CBMS}
Let $\mu$ be a finite  positive discrete measure supported on   $\R$ such that $L^1(\mu)$ contains polynomials.

\ms\no Polynomials are not dense in $L^1(\mu)$  if and only if there exists a balanced zero density subsequence $\L=\{\lan\}\subset \supp \mu$ such that its characteristic sequence $P(\L)=\{p_n\}$ satisfies
$$\exp{p_n}=O(\mu(\{\lan\}))$$
as $|n|\to \infty$.

\end{theorem}

\ms\no Similar statements can be formulated for families of exponential functions in place of polynomials.
In such statements zero density sequences are replaced with sequences of positive density, which
makes them closer related to the regular spectral problems considered in this note. For such results, along with the case of $L^p,\ p\neq 1,$ or Bernstein's spaces, see for instance \cite{Poly, CBMS}.
The case of a general measure can be reduced to the discrete case via some of the standard tools of completeness problems.
For further references and historic remarks see also
\cite{Lub, Meg}.

\ms\no As the reader may have already guessed, the characteristic sequence $p_n$ is nothing else but the sequence of pointmasses of the even operator
$\gamma_n$ in the case when $\L$ is a spectral sequence, i.e., $P(\L)=\Gamma(\L)$. Since any discrete sequence is a spectral sequence for a suitably chosen Krein's canonical system,
one could formulate the last statement using pointmasses of the even operator instead of characteristic sequences.

\ms\no As we saw in the last proof, an alternative way to define the  constants $\gamma_n$ is
$$\gamma_n=1/|F'(\lan)|,$$
where
$$F(z)=\prod \left(1-\frac z{\lan} \right).$$

\end{remark}

\bs\bs\subsection{Hermite-Biehler functions and de Branges spaces}\label{HB} If $F(z)$ is an entire function we
denote by $F^{\#}(z)$ the reflected function $F^{\#}(z)=\bar F(\bar z)$. Note that $F^{\#}=\bar F$ on $\R$.

\ms\no An entire function $F(z)$ belongs to the Paley-Wiener class $PW_a,\ a>0,$
if and only if it is a Fourier transform of a square-summable function
supported on $(-a,a)$. An equivalent definition (the equivalence is established
by the Paley-Wiener theorem) is that $F\in PW_a$ iff
$$ \frac{F(z)}{e^{-iaz}} \in H^2(\C_+), \hspace{1cm}
\frac{F^{\#}(z)}{e^{-i az}} \in H^2(\C_+),$$
where  $H^2(\C_+)$ is the Hardy space in the upper half-plane $\C_+$.

\ms\no  The definition of the de~Branges spaces of entire functions may be
viewed as a generalization of the last definition of the
Paley-Wiener spaces with $e^{-iaz}$ replaced by a more general
entire function. Consider an entire function  $E(z)$ satisfying
the inequality
$$ |E(z)|>|E(\bar{z})|, \ \  z \in \C_{+}.$$ Such
functions are usually called   de~Branges functions. The de~Branges
space $B(E)$ associated with $E(z)$ is defined to be the space of
entire functions $F(z)$ satisfying
$$ \frac{F(z)}{E(z)} \in H^2(\C_+), \hspace{1cm}
\frac{F^{\#}(z)}{E(z)} \in H^2(\C_+).$$ It is a Hilbert space
equipped with the norm $\|F\|_E =\|F/E\|_{L^2(\R)}.$ If $E(z)$ is of exponential type then
all the functions in the de~Branges space $B(E)$ are of
exponential type not greater then the type of $E(z)$. Such a de~Branges
space is called short (or regular) if together with
every function $F(z)$ it contains $(F(z)-F(a))/(z-a)$ for any
$a\in\C$.

\ms\no One of the most important features of de Branges spaces is that they admit a second, axiomatic, definition. Let $H$ be
a Hilbert space of entire functions that satisfies the following axioms:

\begin{itemize}

\item (A1) If $F\in H,\ F(\l)=0$, then $\frac{F(z)(z-\bar\l)}{z-\l}\in H$ with the same norm;

\item (A2) For any $\l\not\in\R$, point evaluation at $\l$ is a bounded linear functional on $H$;

\item (A3) If $F\in H$ then $F^\#\in H$ with the same norm.

\end{itemize}

\ms\no Then $H=B(E)$ for a suitable de Branges function $E$. This is Theorem 23 in \cite{dBr}.

\ms\no Usually, for a given Hilbert space of entire functions it is not difficult to verify the above axioms and conclude
that the space is a de Branges space. It is however a challenging problem in many situations to find a generating
function $E$. This problem can be viewed as an abstract generalization of the inverse spectral problem for
second order differential operators.

\ms\no In the case of regular de Branges spaces, this connection can be illustrated as follows. Given a spectral measure of a regular differential operator one constructs a chain of Hilbert spaces of entire functions, equal to Paley-Wiener spaces as sets with norms inherited from $L^2(\mu)$. The norms
generated by measures corresponding to regular \Sch operators or Dirac systems will be equivalent to the standard norms in $PW_a$, see
for instance \cite{Etudes}. After that, by verifying the axioms one can prove that the resulting spaces are de Branges spaces.
The problem of recovering the  generating functions $E_t$ then becomes equivalent to the inverse spectral problem since these functions
provide a solution to the original equation (system),
thus allowing one to recover the operator.
We will discuss such chains of spaces and relations between $E_t$ and solutions to the \Sch equation in the next section.




\ms\no Recall that an entire function is called real if it is real on $\R$. Any entire function $F$ can be represented as
$F=C+iD$, where $C=\frac 12(F+F^\#), \ D=\frac 1{2i}(F-F^\#)$ are real entire functions. In the case
when $E=A+iB$ is a de Branges function, its real and imaginary parts $A$ and $B$ can be viewed as
analogs of $\sin z$ and $\cos z$ in the standard Fourier settings.

\ms\no We will adopt a common terminology and call de Branges functions with no real zeros Hermite-Biehler functions.
Hermite-Biehler functions form an important subclass of de Branges functions appearing in many applications.
In particular, de Branges functions associated with \Sch equations represent the Hermite-Biehler class.

\ms\no It is well known that $E=A+iB,\ E\neq 0\textrm{ in }\C_+,$ of exponential type is an Hermite-Biehler function,
if and only if the real functions $A$ and $B$ have real alternating (interlacing) zeros.

\ms\no Each de Branges space possess a family of spectral measures, i.e. discrete measures $\mu$ on $\R$ such that  the natural embedding
$B(E)\to L^2(\mu)$ is a unitary operator. One of the natural choices for the spectral measure is the measure $\nu$
defined by the equation
$$H(|E|^{-2}\nu)=B/A,$$
where  $A=\frac 12(E-E^\#)$ and $B=\frac 1{2i}(E-E^\#)$ are the real and imaginary parts of $E$ on $\R$.
The isomorphism $B(E)\to L^2(\nu)$ generalizes the Parseval theorem.

\ms\no For more on de Branges spaces see \cite{dBr, Rom, MPS}.

\bs\bs\subsection{de Branges spaces in \Sch settings}\label{HBS}

\ms\no Now let us return to \Sch equations and the discussion of related analytic functions in Section \ref{functions}.

\ms\no As before, let $u_z(t)$ denote the solution of \eqref{e00} satisfying the Dirichlet boundary conditions at 0.

\ms\no For any $t\in (0,\pi)$ the function $E_t(z)=zu_z(t)+iu_z'(t)$ is an entire function from the Hermit-Biehler class.
The spaces $B(E_t)$ form a chain, i.e., $B(E_t)$ is isometrically embedded into $B(E_s)$ for any $t\leq s$. It follows from our definitions of the \Sch spectral measures $\mu_\pm$ and the above discussion of spectral measures for de Branges spaces that
$\mu_-$ is a spectral measure for $B(E_\pi)$, i.e.,
the  spaces $B(E_t),\ 0\leq t\leq\pi$
are isometrically embedded in $L^2(\mu_-)$.

\ms\no   Similarly, one can consider the 'right-to-left' chain of de Branges spaces $B(F_s)$, $F_s(z)=zv_z(s)-iv'_z(s)$,
satisfying $B(F_t)\supset B(F_s)$ for $t\leq s$. The measure $\mu_+$ is a spectral measure for $B(F_0)$ and
all $B(F_s)$ are isometrically embedded into $L^2(\mu_+)$.

\ms\no The well-known asymptotics of the solutions $u_z$ and $v_z$ imply that both chains are regular.

\ms\no It is well known that in the case of regular \Sch operators the de Branges spaces $B(E_t)$ are equal to
Paley-Wiener spaces $PW_{t/\pi}$ as sets (with equivalent norms). Indeed, it follows from the asymptotics \eqref{asymp1} that
the norm of $L^2(\mu_-)$ is equivalent to the standard norm in each $PW_{t/\pi},\ 0\leq t\leq \pi$, see for instance \cite{OS}. Verifying the axioms we can show that each $PW_{t/\pi}$ equipped
with the $L^2(\mu_-)$-norm is a de Branges space. By the uniqueness of a regular de Branges chain in each $L^2$-space over
a Poisson-finite measure, the new chain has to coincide with $B(E_t)$.

\ms\no Similarly, $B(F_s)=PW_{(\pi-s)/\pi}$ as sets. For more on the basics of de Branges theory see \cite{dBr, Rom}. For the particular
case of \Sch equations see \cite{Dym, Remling}.

\ms\no One can consider the above to be a 'shortcut' definition of de Branges chains corresponding to \Sch equations, i.e., one can define $B(E_t)$ to
be the Paley-Wiener space $PW_{t/\pi}$ equipped with an equivalent norm from $L^2(\mu_-)$. Establishing the equivalence of norms
is a good exercise. Let us remark that the study of more general measures providing equivalent norms for the Payley-Wiener spaces is yet another deep classical problem of analysis, see for instance \cite{OS}. After the equivalence of norms is established, verifying the axioms is straight-forward.
With a bit of technical work one can also verify that $E_t(z)=zu_z(t)+iu_z'(t)$ is an Hermite-Biehler function and proceed with the first definition,
which leads to the same chain of spaces.

\ms\no The equality of de Branges spaces to Paley-Wiener spaces as sets is one of the tools of the Gelfand-Levitan theory
in the area of spectral problems for \Sch equations and Dirac systems, see for instance \cite{GL, LS}.
In these papers the differential operator is defined on the half-axis $[0,\infty)$ and restrictions of the operator to $[0,t)$
give rise to a chain of de Branges spaces $B(E_t)$, each equal to the Paley-Wiener space $PW_t$ as sets. After that the well-studied structure
of the $PW$-spaces allows one to write integral equations for the Fourier transform of the spectral measure of the operator and obtain results
relating its properties to the properties of the potential, see also \cite{Etudes, Rem}.

The same property can be viewed
from perturbational point of view: for small (summable) potentials the de Branges chain coincides, as sets, with the Paley-Wiener
chain of spaces, which is the de Branges chain for  the free operator ($q=0$). In this regard one may ask for what general canonical systems (\Sch equations, Dirac systems)
the corresponding de Branges chains coincide as sets. One of the versions of this problem is to characterize Hamiltonians such that de Branges spaces are Paley-Wiener spaces as sets (we know  this is true for regular Schrodinger and Dirac operators).

\ms\no While Paley-Wiener spaces appear only in the case of a regular left end point the method applies in the singular case as well (at least in the  compact resolvent case)  but the "model" de Branges spaces will no longer be Paley-Wiener spaces.

\bs\bs\subsection{Hermite-Biehler functions for \Sch operators}\label{HBSthm}

\ms\no Not every Hermite-Biehler function can be obtained from a \Sch equation in the way described in the last section. Characterization of such functions
is important for applications, as we will illustrate in Section \ref{sHorvath}. Such a characterization was recently obtained in \cite{BBP}. Here we
present a version of the same result which will be more convenient for our purposes.

\ms\no Consider the 'backward shift operator' $S^*$ on $PW_a$ defined as $S^*f=\frac {f-f(0)}z$. This is a bounded operator
with a dense image.  We will denote the image of $S^*$ by $PW_a'$. Functions from $PW_a'$ appear naturally in relation with \Sch operators.

\begin{theorem}\label{tHBS}
An entire function $E=A+iB$ of Hermite-Biehler  class  corresponding to a \Sch equation from $\SS^2$ satisfies
\begin{equation}A=\sin z + f\textrm{ and }B=\cos z +g,\label{eAB}\end{equation}
where $f$ is an odd function from $PW_1'$ and $g$ is an even function from $PW_1'$ such that $f=S^*F, g=S^*G$ for some $F,G\in PW_1,\ F(0)=G(0)$. Moreover,
for any such $f, g\in PW_1'$ with $F(0)=G(0)$
there exists $\e>0$ such that for any $c_1\in \R,\ |c_1| <\e$ and $c_2\in\R,\ |c_2| <\e$
the function $E=A+iB$ with
\begin{equation}A=\sin z +c_1 f\textrm{ and }B=\cos z + c_2 g\label{eAB2}\end{equation} is an Hermite-Biehler function corresponding to
a \Sch equation from $\SS^2$.

\end{theorem}

\begin{proof}

\ms\no If the real entire functions $A$ and $B$ have the form as in \eqref{eAB2} with small enough $c_k$ then their zero sets satisfy
the asymptotics \eqref{e000a}, \eqref{e000}. Hence these sequences form spectra $\sigma_{DD}$ and $\sigma_{ND}$ for an operator $L\in \SS^2$. Then the Hermite-Biehler
functions $A',\  B'$ obtained from $L$ must coincide with $A, \  B$ up to constant multiples because the zero sets of $A'$ and $B'$ coincide with those of $A$ and $B$ correspondingly.
 It follows from the well-known asymptotics for the solutions of regular \Sch equations $u_z$, see for instance Theorem 3, p. 13, of \cite{Tru}, that the functions are equal.

\ms\no In the opposite direction, if $E$ is obtained from a \Sch equation then the zero sets $\{\lan\}$ of $A$ and $\{\eta_n\}$ of $B$ satisfy \eqref{e000a}, \eqref{e000}.
Representing each function as an infinite product we obtain that $A(\pi n)=\frac {c+a_n}{|n|+1}$ and $B(\pi (n+\frac 12))=\frac {c+b_n}{|n|+1}$
for some $c\in\R, a_n, b_n\in l^2$. By Parseval's theorem, there exist unique functions  $F,G\in PW_1$ such that $F(\pi n)=n A(\pi n) -c$ and $G(\pi (n+\frac 12))=n B(\pi (n+\frac 12)) -c$. Put $f=S^*F$ and $g=S^*G$.

\ms\no Once again, from the asymptotics of solutions we see
that $A(z)-\sin z= zu_z-\sin z \in PW_1$. Since the functions from $PW_1$ are uniquely determined by their values at $\pi n$ and
$A(\pi n)-\sin \pi n=f(\pi n)$, $A=\sin z +f$. Similarly, $B=\cos z +g$.

\end{proof}

\begin{remark}\normalfont Note that, as follows from the proof, the constant $\e$ can always be chosen as $1/\max(||f||_2, ||g||_2)$.

\ms\no The statement of the theorem implies the asymptotics of the spectra \eqref{e000a}, \eqref{e000} and of the pointmasses of
spectral measures \eqref{asymp1}, \eqref{asymp2} and is in fact equivalent to those asymptotics via Parseval's theorem
and the equivalence of the corresponding de Branges chains to $PW$-chains.

\end{remark}

\bs\bs\subsection{Krein spectral shift}\label{Kshift}

Note that if $\mu$ is a positive measure, then $H\mu$ has a positive imaginary part in $\C_+$.
Hence $\log H\mu$ is a well defined holomorphic function in $\C_+$. Since $H\mu$ is real
on $\R$ one may consider a branch of $\log H\mu$ whose imaginary part takes values between $0$ and $\pi$
in $\C_+$ and define the Krein spectral shift function
$k(x)=k_\mu(x)$  corresponding to $\mu$ as $k=\Im \log H\mu$ on $\R$.

\ms\no Note that if $\mu$ is a positive singular measure, then $k$ is a function which takes only two values, $0$ and $\pi$, on $\R$. If $\mu$ is a discrete measure concentrated on
a sequence $\L$ then $k=\pi$ on  a union of disjoint intervals $(\lan, \lan +c_n)$ and zero elsewhere. The lengths of the intervals, $c_n$, are determined by the pointmasses of $\mu$ at $\lan$.

\ms\no More general Krein spectral shift functions appear in perturbation problems, see for instance \cite{Krein}. The above definition
corresponds to the simplest case of the theory, when the perturbed operator is a self-adjoint operator with the spectral measure
$\mu$ and the perturbation is self-adjoint of rank one. In this case $\lan$ form the spectrum of the original operator and $\lan+c_n$
of the perturbed one, see for instance \cite{PoltKrein}.

\ms\no As follows from the above definition, $k=k_\mu$ satisfies
\begin{equation}H\mu(z)=\exp\left(Hk(z)+c     \right)\label{e42}\end{equation}
for some real constant $c$.  The function $k$ and the constant $c$ are uniquely determined by $\mu$. As seen from the above equation,
$c=\log |H\mu(i)|$. In the opposite direction, for a fixed function $k$ there is a family of positive measures satisfying
\eqref{e42} differing from each other by a constant multiple $e^c$.

\begin{lemma}\label{l1}
 Let $\mu=\sum \alpha_n\delta_{\lan}$ be a  positive Poisson-finite measure supported on  a separated sequence $\L$, $\alpha_n=o(1/\lan)$. Let $k=k_\mu$ be the corresponding
Krein spectral shift. Suppose that $k=\pi$ on $\cup (\lan, \lan+\e_n)$. Then either
\begin{equation}\e_n=O(\l_n\alpha_n)\textrm{ as }n\to\infty\label{e41}\end{equation}
or
\begin{equation}( \lan-\l_{n-1}-\e_{n-1})=O(\l_n\alpha_n)\textrm{ as }n\to\infty.\label{e43}\end{equation}



\end{lemma}

\begin{proof}

\ms\no Let $\mu$ and $k$ be like in the  statement. First let us show that either $\e_n\to 0$ or $( \lan-\l_{n-1}-\e_{n-1})\to 0$.
Indeed, if neither holds, then there exists $\delta>0$ such that for infinitely many $n$ both $\e_n>\delta$ and
$( \lan-\l_{n-1}-\e_{n-1})>\delta$. Notice that for such $n$, $\alpha_n\gtrsim 1/|\lan|$.

\ms\no Suppose now that $\e_n\to 0$. The statement follows from
$\res  (e^{H\chi_{(\lan,\lan+\e_n)}})|_{\lan}\asymp \e_n$ and
$$H\chi_{\cup_{k\neq n}(\lan, \lambda_k+\e_k)}(\lan)=O(\log \lan).$$
The second case can be treated similarly.

\end{proof}

\ms\no Let $\L$ and $\HH$ be sequences like in \eqref{e000a}, \eqref{e000}. As was discussed before, such sequences present two spectra of a \Sch operator
$L$.
We will call the Krein function $k$ which is equal to $\pi$ on $\cup_{n\in\N} (\l_{n-1},\eta_n)$
on $\R_+$ and on $\cup_{n\in -\N}(\lan,\eta_n)$ on $\R_-$ and to $0$ elswhere
 the Krein spectral shift function of $L$.

\ms\no It can be easily seen from our definitions
that $k$ is the Krein spectral shift function for $\mu_+$, i.e., the function and the measure
satisfy \eqref{e42}. Note that because of the asymptotics \eqref{asymp1}, the constant in \eqref{e42} has to be 0.
Hence the Krein spectral shift uniquely determines $\mu_\pm$. Making an obvious observation that $k$, in its turn, is uniquely determined
by its positive and negative jumps, i.e., by the two sequences $\sigma_{DD}$ and $\sigma_{ND}$, and applying Marchenko's uniqueness
result, we deduce Borg's two-spectra Theorem \ref{Borg}.
This is the proof given by Donoghue in \cite{D}. For further extensions of the same method, including multi-dimensional analogs of the Krein function
see \cite{SG3, GT}.

\bs\section{Completeness, Gap and Type Problems}

\bs\subsection{Beurling-Malliavin densities and the radius of completeness}\label{BM}

\ms\no If $\{I_n\}$ is a sequence of disjoint intervals on $\R$, we call it short if
$$\sum\frac{|I_n|^2}{1+\dist^2(0,I_n)}<\infty$$
and long otherwise.

 \ms\no If $\L$ is a sequence of real points define its exterior Beurling-Malliavin (BM) density (effective BM density)
as

$$D^*(\L)=\sup\{ d\ |\ \exists\textrm{ long  }\{I_n\}\textrm{ such that }\#(\L\cap I_n)\geqslant d|I_n|\},\ \forall n\}.$$

\ms\no For a non-real sequence its density can be defined as $D^*(\L)=D^*(\L')$ where $\L'$ is a real sequence $\lan'=\frac 1{\Re \frac 1{\lan}}$,
if $\L$ has no imaginary points, or as $D^*(\L)=D^*((\L+c)')$, with a properly chosen real constant $c$, otherwise.

\ms\no  A dual definition is used to introduce the interior BM density:
$$D_*(\L)=\inf\{ d\ |\ \exists\textrm{ long  }\{I_n\}\textrm{ such that }\#(\L\cap I_n)\leqslant d|I_n|\},\ \forall n\}.$$

\ms\no Both densities play important role in Harmonic Analysis by appearing in a number of fundamental results. Their applications were recently extended
in to the area of spectral problems for differential operators via the methods discussed in this paper. As an example, let us recall the original
appearance of $D^*$ in the solution of a completeness problem
by Beurling and Malliavin.

\ms\no For any complex sequence
$\L$ its
radius of completeness is defined as
$$R(\L)=\sup\{ a\ |\ \EE_\L=\{e^{i\l z}\}_{\l\in\L} \textrm{ is complete in }L^2(0,a)\}.$$

\ms\no One of the fundamental results of Harmonic Analysis is the following theorem (see \cite{CBMS} for history and further references).

\begin{theorem}[Beurling and Malliavin, around 1961, \cite{BM1, BM2}]\label{bigBM}
Let $\L$ be a discrete sequence. Then
$$R(\L)=2\pi D^*(\L).$$
\end{theorem}

\ms\no We will return to this result and the exterior density when we discuss Horvath' theorem in Section \ref{sHorvath}. The other density, $D_*(\L)$, which has recently made a new appearance in the area of the Gap and Type Problems (see \cite{Polya, Gap, Type, CBMS})
will be used in our statements below. Note that for subsequences of  sequences close to arithmetic progressions, the two densities are  related to each other in a rather simple way.  In particular for sequences $\L$ satisfying \eqref{asymp1} or \eqref{asymp2}, for any $\Phi\subset \L$,
$$D^*(\Phi)+D_*(\L\setminus\Phi)=D^*(\L)=D_*(\L) =1.$$

\bs\bs\subsection{Spectral gaps, types and sign changes}\label{Gap}

\ms\no We will call a lower semi-continuous function $W:\R\to [1,\infty]$ a weight on $\R$. We will say that a measure $\mu$ on $\R$ is $W$-finite
if
$$||\mu||_W=\int Wd|\mu|<\infty.$$
Note that $W$-finite measures are forced to be supported on the subset of $\R$ where $W$ is finite.

\ms\no If $W$ is a weight we define its type $T_W$ as
$$T_W=\sup \{ a|\textrm{ $\exists$ $W$-finite non-zero measure $\mu$ with a spectral gap $[-a,a]$}\}.$$
This is one of seveal equivalent definitions for $T_W$, see \cite{CBMS}. The type of $W$ is used in applications such as problems on completeness
of exponential systems in $L^p$ and Bernstein's spaces.

\ms\no If $\nu$ is a real measure on $\R$ we will denote by $\nu^\pm$ its positive and negative parts.
We will say that $\nu$ has a spectral gap $(-a,a)$ if $\int fd\nu=0$ for all $f\in PW_a\cap L^1(|\nu|)$.
Note that for finite measures this property coincides with $\hat{\nu}=0$ on $(-a,a)$.

\ms\no If $A$ and $B$ are two closed subsets of $\R$,  let
$\MM^W_a(A,B)$ be the class of all non-zero $W$-finite real measures $\sigma$ with a spectral gap $[-a,a]$ such that $\supp \sigma^+\subset A$ and $\supp \sigma^-\subset B$.


\ms\no The following statement is a combination of Lemmas 13 and 16 from \cite{Det}.

\begin{lemma}\label{MiP}
If $\MM^W_a(A,B)$ is non-empty it contains a discrete measure $\nu$, whose positive and negative parts
have interlacing supports.

\end{lemma}

\ms\no The following version of the Type theorem is Theorem 36 from \cite{CBMS}

\begin{theorem}\label{Bmain}
$$T_W=\pi\sup \left\{d\ : \sum\frac{\log W(\lan)}{1+\lan^2}<\infty\textrm{ for some $d$-uniform sequence }\L \right\},$$
if the set is non-empty, and $0$ otherwise.
\end{theorem}

\ms\no The notion of $d$-uniform sequences first appeared in \cite{Gap}. In this paper we do not need a full definition
of a $d$-uniform sequence, since the sequences we are concerned with are separated. For a separated sequence $\L$,
i.e., a discrete sequence such that $|\lan-\l_{n-1}|>c>0$ for all $n$, $\L$ is $d$-uniform iff $D_*(\L)=d$, see
for instance Example 1 on page 27 of \cite{CBMS}. This implies

\begin{corollary}\label{type}
$$T_W\geq \pi\sup \left\{d\ : \sum\frac{\log W(\lan)}{1+\lan^2}<\infty\textrm{ for some separated }\L,  \  D_*(\L)=d\right\}.$$

\end{corollary}

\ms\no
From one of the results of \cite{Det} (or Theorem 17, Chapter 4, in \cite{CBMS}) we deduce

\begin{corollary}\label{forTeven} If $W$ is a weight, $A$ is a separated sequence and $B$ is any closed set then
$$\sup\{a|\ \MM^W_a(A,B)\neq \emptyset\}\leq 2\pi D_*(A).$$

\end{corollary}

\ms\no
We denote $\log_-x=\max(0,-\log x).$
We will also need the following

\begin{corollary}\label{l2} Let $X=\{x_n\}$ be a separated sequence of real numbers.
Let $\tau=\sum c_n\delta_{x_n}$ be a discrete measure with spectral gap $(-a,a)$. Then
for any $d<a$ there exists $\Phi\subset X$, $\pi D_*(\Phi)>d$, such that
$$\sum_{x_n\in\Phi} \frac{\log_- c_n}{1+n^2}<\infty.$$

\end{corollary}

\begin{proof}
Let $W$ be the weight defined as $W(x_n)=\frac 1{|c_n|(1+x^2)}$ on $X$ and as infinity elsewhere. Then $|\tau|$ is
a $W$-finite measure of type at least $a$. The rest follows from Corollary \ref{type}.

\end{proof}

\bs\bs\subsection{A new proof of Horvath' theorem}\label{sHorvath}

\ms\no The following theorem is one of the main results of \cite{Horvath}. Here we formulate it
in an equivalent form using our version of the $m$-functions (after the square root transform).

\begin{theorem}\label{Horvath} \cite{Horvath} Let $\L=\{\lan\}$ be
a sequence of distinct non-zero complex numbers, $0\leq a \leq 1$.
The following statements are equivalent:

\ms\no 1) For any \Sch operator $L\in \SS^2$, if $\ti L\in \SS^2$ is such that $q=\ti q$ on $(0,a\pi)$
and $m_-=\ti m_-$ on $\L$, then $L$ and $\ti L$ coincide identically.

\ms\no 2) The system of exponentials $\{e^{i\gamma z}| \gamma \in \L\cup\{*,*\}\}$ is complete in \newline $L^2(0,(2-2a)\pi)$.
\end{theorem}

\ms\no The notation $\{*,*\}$ in the statement stands for any two points in $\C\setminus\L$.
A version of the above theorem is proved in \cite{Horvath}
for all $1\leq p\leq \infty$. In this paper we treat only
the case $p=2$, although a similar argument can be applied to other $p$.
A simple proof for the "un-mixed" case $a=0$ is given in \cite{BBP}.
Here we show how to deduce the full statement from Theorem \ref{tHBS} and Lemma \ref{l5}.

\ms\no If $f,g$ are two functions from the Smirnov class in $\C_+$, we say that $f$ is divisible by $g$ if $f/g$ again
belongs to the Smirnov class in $\C_+$. We denote by $PW^{even}_a$ the subset of $PW_a$ consisting of functions which are
even on $\R$.

\begin{proof} Let $L,\ti L$ be the operators with $q=\ti q$ on $(0,a\pi)$. By  Corollary \ref{c01} in the next section, the function

\begin{equation}H\mu-H\ti \mu=\frac AB -\frac{\ti A}{\ti B}=\frac {A\ti B - \ti A B}{\ti B B}\label{eDrob}\end{equation}
is divisible  by $e^{i2az}$ in the upper half-plane. Since $\ti B B$ grows like $e^{i2az}$ along $i\R_+$,
the numerator grows no faster than $e^{i(2-2a)z}$. Together with Theorem \ref{tHBS} we obtain that
$A\ti B - \ti A B=f$ is an odd function such that  $zf\in PW_{2-2a}$. If $m=\ti m$ on $\L$,
$zf$ is zero on $\L\cup \{*,*\}$ (recall that it has a double zero at 0). If the system of exponentials is
complete then $\L\cup \{*,*\}$ is a uniqueness set for $PW_{2-2a}^{even}$ and we obtain that $L\equiv \ti L$.

\ms\no Let $L\in \SS^2$ be any operator and let $E=A+iB$.
If the exponential system is incomplete, $\L\cup \{*,*\}$ is not a uniqueness set and there exists  a real $F\in PW^{even}_{2-2a}$
vanishing on $\L$ and with an additional double zero at 0. Then using the second part of Theorem \ref{tHBS}, $\ti A=A+cF/z$
with sufficiently small $c$ and $\ti B=B$ will produce  $\ti L$ with the same values of the $m$-function on $\L$.
Notice also that the function  in \eqref{eDrob} will be divisible by $e^{i2az}$ in $\C_+$, and by Corollary \ref{c01} below, $q=\ti q$ on $(0,a\pi)$.
\end{proof}

\ms\no Horvath' theorem establishes equivalence between mixed spectral problems for \Sch operators and the Beurling-Malliavin problem
on completeness of exponentials in $L^2$ spaces discussed in Section \ref{BM}.
Combining Theorems \ref{Horvath} and \ref{bigBM} we obtain the following statement, which is the sharpest possible result formulated in terms
of the density of the defining sequence $\L$.

\begin{corollary}
Let $\L=\{\lan\}$ be
a sequence of distinct non-zero complex numbers, $0\leq a \leq 1$.
The following statements are equivalent:

\ms\no 1) Any two Schroedinger operators $L$ and $\ti L$, such that $q=\ti q$ on $(0,d\pi)$ for some $d>a$
and $m_-=\ti m_-$ on $\L$, coincide identically.

\ms\no 2) $  D^*(\L)\geq 1-a$.

\end{corollary}

\bs\section{Applications to mixed spectral problems}

\bs\subsection{Fourier gaps and Schr\"odinger operators}\label{FSGap}
 We will formulate the results of this section for a slightly broader class $\SS^1$ of \Sch operators with summable potentials.

\begin{lemma}\label{l5}
Let $L,\ti L\in \SS^1$. Then  $q=\ti q$ on $(0,a\pi)$ iff $\mu_--\ti\mu_-$ has a spectral gap  $(-2a,2a)$. Similarly,
$q=\ti q$ on $(b\pi,\pi)$ iff $\mu_+-\ti\mu_+$ has a spectral gap  $(-2(1-b),2(1-b))$.
\end{lemma}

\begin{proof}
If $q=\ti q$ on $(0,a\pi)$ then $E_t=\ti E_t$,  $t\in (0,a\pi)$. Since $B(E_t)= B(\ti E_t)=PW_{t/\pi}$ as sets,
for any $f,g\in PW_{t/\pi}$
$$ \int f\bar g d\mu_-=<f,g>_{B(E_t)}=<f,g>_{ B(\ti E_t)}= \int f\bar g d\ti\mu_-, $$
for $t\in (0,a\pi)$. It follows that the measure $\mu_--\ti\mu_-$ annihilates any function of the form
$f\bar g$,  where $f,g\in PW_{t/\pi}$, for for $t\in (0,a\pi)$. Next we notice that linear combinations of the products $f\bar g$
are dense in
$PW_{2t/\pi}$.
Indeed, any function in $PW_{2t/\pi}$ can be represented as $e^{i\frac t{\pi}z} g_1+g_2 + e^{-i\frac t{\pi}z} g_3$,
for some $g_k\in PW_{t/\pi}$. By approximating unit pointmasses at $0$ and $\pm t/\pi$ with $L^2$-functions and taking
their Fourier transforms in place of the exponentials, we approximate any such sum with linear combinations of products $f\bar g$.
  Thus $\mu_--\ti\mu_-$ annihilates $PW_{2t/\pi}$ and therefore has the stated spectral gap.
The argument can be reversed to prove the opposite implication and the second statement follows similarly.
\end{proof}

\begin{remark}\normalfont It is a well known statement in the area of the gap problem that a real measure $\mu$ has a spectral gap $(-a,a)$
if and only if $H\mu(iy)=O(e^{-ay})\textrm{ as }y\to\infty,$ see for instance Lemma 4.5 in \cite{Polya}. In this statement $O(e^{-ay})$ can be replaced with
$o(e^{-y(1-\e)a})$ and the positive imaginary half-axis with negative.
\end{remark}

\begin{lemma}\label{l6} Let $L,\ti L\in \SS^1$ and let $k,\ti k$ be the corresponding Krein functions. Then $q=\ti q$ on $(0,a\pi)$ iff the measure
$(k-\ti k) dx$ has a spectral gap  $(-2a,2a)$.
\end{lemma}

\begin{proof}
Recall that $\mu_--\ti\mu_-$ has a spectral gap  $(-2a,2a)$ if and only if $H(\mu_--\ti\mu_-)(iy)=O(e^{-2ay})$ as $y\to\infty$.
But
$$H(\mu_--\ti\mu_-)=H\mu_-\left(1-\frac{H\ti\mu_-}{H\mu_-}\right)=H\mu_-\left(1-e^{H(\ti k-k)}\right).$$
Since $\mu_-$ is a positive measure, its Herglotz integral decays no faster than polynomially on $i\R_+$. Hence
$H(\ti k-k)$ must decay to zero faster than $e^{-2(1-\e)ay}$.
\end{proof}

\ms\no As was mentioned above, the presence of spectral gap can be equivalently reformulated in terms of decay along $i\R$:

\begin{corollary}\label{c01}Consider two operators $L$ and $\ti L$ from $\SS^2$ and let $k,\ \ti k$ be their Krein functions. Then
the potentials $q$ and $\ti q$ are equal on $(0,a\pi)$ iff
$$H(k-\ti k)(iy)=O(e^{-2ay})\textrm{ as }y\to\infty$$
iff
$$H(k-\ti k)(iy)=o(e^{-2(1-\e)ay})\textrm{ as }y\to\infty$$
for any $\e>0$.

\end{corollary}

\begin{remark}\normalfont It follows from the proof of Lemma \ref{l5} that  $q=\ti q$ on $(0,a\pi)$ iff
$\nu_- -\ti \nu_-$  has a spectral gap  $(-2a,2a)$, where $\nu_-$ is the spectral measure
corresponding to the Neumann boundary condition at 0.  The same fact follows from Lemma \ref{l6} because $\pi - k$ and $\pi - \ti k$ are
Krein functions for $\nu_-$ and $\ti \nu_-$.

\ms\no More generally, any self-adjoint boundary condition at 0 or $\pi$ can be used in this statement with the same proof.
\end{remark}
\ms\no  Together with the above this produces the following statement, which is one of the main results of \cite{S}. It is formulated in \cite{S} in an equivalent form,
using a different $m$-function (before the square root transform).

\begin{corollary}\label{c1}
The potentials $q$ and $\ti q$ coincide on $(0,a\pi)$ iff the $m$ functions satisfy
$$m_+(iy)-\ti m_+(iy)=O(e^{-2ay})\textrm{ as }y\to\infty.$$
\end{corollary}

\bs\bs\subsection{The size of uncertainty}\label{2int}

Recall that for an operator $L$ we denote by $\sigma_{DD}=\{\lan\}$ and $\sigma_{ND}=\{\eta_n\}$ its spectra
after the square root transform. We will enumerate
the sequences as described in Section \ref{spectra}.

\ms\no Let $I=\{I_n\}_{n\in\Z}$ be a sequence of intervals on $\R$. Consider  the set $\SS_I$ of those operators $L\in \SS^2$ for which
$\sigma_{DD}$ and $\sigma_{ND}$ lie in the union of $I_n$. Following our discussion in the introduction,
we can ask what part of the necessary spectral information we are given by this inclusion condition?
Let us define the size of uncertainty for the sequence of intervals $I=\{I_n\}$ as the number $U(I)$
equal to the infimum of $a$ such that knowing the potential of an operator $L\in\SS_I$ on $(0,a\pi)$
one can recover $L$ uniquely. Theorem \ref{t0} below gives the following formula for the size of uncertaity:
$$U(I)=\pi \sup \left\{D_*(\Phi)\ :\ \  \sum_{n\in\Phi}\frac{\log_-|I_n|}{1+n^2}<\infty\right\}.$$

\ms\no Let us now make our statements more precise.

\begin{theorem}\label{t0}
Let $\{\e_n\}_{n\in\N}$ be positive numbers,  $a\in [0,1)$. TFAE

\ms\no 1) Any two Schr\"odinger operators $L$ and $\ti L$   satisfying
\begin{equation}|\lan-\ti\lan|<\e_{2n},\ |\eta_n-\ti\eta_n|<\e_{2n+1}, \ n\in\N,\label{e1}\end{equation}
 and $q(x)=\ti q(x)$ on $(0,d\pi), d>a$ must coincide identically, i.e.,  $q=\ti q$ a. e. on $(0,\pi)$.

 \ms\no 2) Any sequence of distinct integers $\Phi\subset \Z$ such that
 \begin{equation}\sum_{n\in\Phi\cap\N}\frac{\log_-\e_n}{1+n^2}<\infty,\label{e2}\end{equation}
satisfies $ D_*(\Phi)\leq a$.

\end{theorem}

\begin{proof}
By Corollary \ref{c1}, $q(x)=\ti q(x)$ on $(0,d\pi)$ for some $d>a$ iff the measure $(k(x)-\ti k(x))dx$ has a spectral gap $(-2d,2d)$.

\ms\no Suppose now that \eqref{e2} is satisfied only for $\Phi, D_*(\Phi)\leq a$, but there exist $L,\ti L$ satisfying
\eqref{e1} with $q(x)=\ti q(x)$ on $(0,d\pi), d>a$. Notice that the intervals $(\lan,\ti\lan), (\eta_n, \ti\eta_n)$ are
disjoint when $|n|$ is large; WLOG we will assume that they are disjoint for all $n$. The union of these intervals
constitutes the support of the function $l=k-\ti k$.

\ms\no Define the weight $W(x)$ as $W(x)=\e_n^{-1}/(1+x^2)$ on $(\lan,\ti\lan)\cup (\eta_n, \ti\eta_n)$ and continuously
on the rest of $\R$. Denote $A=\{l>0\}, B=\{l<0\}$. Then the set $M^W_d(A,B)$ is non-empty because it contains
the measure $l(x)dx$. By Lemma \ref{MiP}, it contains a discrete measure $\nu$ whose positive
and negative parts alternate between $A$ and $B$, i.e., each of the intervals $(\lan,\ti\lan), (\eta_n, \ti\eta_n)$ contains at most one
pointmass of $\nu$, with all positive masses contained in $A$ and all negative ones in $B$.

\ms\no Let $X=\{x_n\}$ be the support of $\nu$. Define the weight function $V(x)$ to be infinity on $\R\setminus X$
and to be equal to $W(x)$ at the points of $X$. Then $T_V\geq 2d$ because $\nu$ is a $V$-finite measure which
has a spectral gap $(-2d,2d)$. On the other hand, $T_V\leq 2a$ by \eqref{e1}, \eqref{e2} and Corollary \ref{type}.

\ms\no In the opposite direction, suppose that \eqref{e2} is satisfied  for some $\Phi, D_*(\Phi)> a$. Let us show that then there exist
$L,\ti L$ satisfying
\eqref{e1} with $q(x)=\ti q(x)$ on $(0,a\pi)$. Let $V, \ V(0)=1,$ be the even weight which is infinite outside of $\pi\Z/2$ and
equals to $\e_{2n}^{-1}/(1+x^2)$ at $\pi n$ and to $\e_{2n+1}^{-1}/(1+x^2)$ at $\pi(n+1/2)$ for $n\in\N$. By \eqref{e2} and Corollary \ref{type}, $T_V\geq 2a$. Hence, there exists
a $V$-finite measure $\nu$ supported on $\pi\Z/2$ with the spectral gap $(-2a,2a)$.
WLOG $n^2\e_n\in l^2$ and the measure $\nu$ is even. Let $k_\pm$ be the Krein spectral shifts for the positive and negative parts of $\nu$, $\nu_\pm$. Then $k_+$ is supported
on the union of intervals $(n_k/2,n_k/2+\e^+_k)$ and $k_-$ is supported
on the union of intervals $(m_k/2,m_k/2+\e^-_k)$. WLOG $\nu_-$ has a pointmass at $0$. Then $n_0=0$. \

\ms\no By Lemma \ref{l1} the pointmasses and the Krein functions of $\nu_\pm$ satisfy either \eqref{e41} or \eqref{e43} near $+\infty$. We
will assume that both satisfy \eqref{e41}: other cases can be treated similarly. In this case $n_k\e_k^\pm\in l^2$.

\ms\no Define the spectral sequences $\L=\{\lan\},\ti\L=\{\ti\lan\}, \HH=\{\eta_n\}, \ti\HH=\{\ti\eta_n\}$ of $L$ and $\ti L$ in the following way.


\ms\no Since  $\e^\pm_k$ tend to zero, for large enough $C$ the supports of $k_\pm$ become disjoint on $(C,\infty)$.
We will also assume that $C$ does not belong to either of the supports. Denote $l=k_+-k_-$.
The spectral sequences $\L,\HH,\ti \L, \ti \HH$ will be made of points where the function $l$ makes the jumps, so
that for the Krein functions $p$ and $\ti p$ of $L$ and $\ti L$ we had $l=p-\ti p$.

\ms\no To achieve the last equation, first distribute all the jumps of $l$ at integer points of $(C,\infty)$ as follows:   add  $x$ to $\L$ if $l=k_+-k_-$ has a positive jump at $x$ or add it to $\ti\L$ if
$l$ makes a negative jump at $x$. Next, if $x$ is an integer jump point of $l$ which was added to $\L$, add the next jump of $l$ to $\ti \L$ and vice versa. If $l$ does not make a jump at an integer point $x>C$, add it to both $\L$ and $\ti L$.

\ms\no As far as half-integer jumps, if $x>C, x\in \Z+1/2$, then  add  $x$ to $\HH$ if $l=k_+-k_-$ makes a negative jump at $x$ or add it to $\ti\HH$ if
$l$ makes a positive jump at $x$. If $x$ was added to $\HH$, add the next jump of $l$ to $\ti \HH$ and vice versa. If $l$ does not make a jump at a half-integer point $x>C$, add it to both $\HH$ and $\ti \HH$.

\ms\no Distribute the points similarly on $(-\infty, -C)$, so that the sequences $\L, \ti \L, \HH$ and $\ti \HH$ are even.

\ms\no On $[0,C)$ assign the points $n_k$ to $\L$, $n_k+\e^+_k$ to $\HH$, $m_k$ to $\ti\L$ and $m_k+\e^-_k$ to $\ti\HH$. On
$(-C,0)$ assign $-x$ to the same sequence as $x$.
Note that
this way we keep the sequences even and assign the same number $c_1$  of points to $\L, \ti L$ and $c_1-1$ points to $\HH,\ti \HH$. If $c_1$ is less than $c_2=\#(\Z\cap [0,C))$, which is the case if
$\nu$ skipped some of the points of $\Z/2$, we choose an interval inside $[0,C)$ where $l$ is zero and choose any $2(c_2-c_1)$ points on that
interval. After that we assign those points to $\L$ and $\HH$, so that all points are alternating between the sets. Those points put in $\L$ should also be added to $\ti\L$
and those put in $\HH$ to $\ti\HH$. Note that WLOG $[0,C)$ has a subinterval where $l$ is zero, otherwise we can just increase $C$. On $(-C,0)$ assign
the symmetric points to keep all four sequences even.

\ms\no Clearly, the sequences satisfy the asymptotics \eqref{e000a}, \eqref{e000}, and therefore define two \Sch operators $L$ and $\ti L$. We claim
that $q(x)=\ti q(x)$ on $(0,a\pi)$.

\ms\no Indeed, by our construction for the Krein spectral shifts $p$ and $\ti p$ of $L$ and $\ti L$ we have
$$H(p-\ti p)(iy)=H(k_+-k_-)(iy)=\const\frac{H\nu_+}{H\nu_-} =o(e^{-2(a-\e)y}).$$
WLOG we can assume that $H\nu(i)=0$. Otherwise, take any two zeros $a, -a \in \R$ and
consider the function $\frac{z^2+1}{z^2-a}H\nu$. It is not difficult to show that
the last function is again a Herglotz integral  $H\left(\frac{z^2+1}{z^2-a}\nu\right)$, which
decays exponentially along $i\R$, hence the measure $\frac{z^2+1}{z^2-a}\nu$ has the desired spectral gap.
Then the constants in the formula \eqref{e42} for $k_\pm$ coincide and
$\const$ in the last equation is 1. By Lemma \ref{l6}, it implies the statement.

\end{proof}

\ms\no In particular we obtain the following 'uncertainty' version of Borg's theorem

\begin{corollary} Let $\{I_n\}$ be a sequence of intervals on $\R$, $U=\cup I_n$. TFAE

\ms\no 1) The condition \begin{equation}\sigma_{DD}\cup \sigma_{ND}\subset U\label{eUP}\end{equation}
together with the values of the potential $q$ on $(0,\e)$ for any $\e>0$ determines an operator $L\in \SS^2$
uniquely

\ms\no 2) For every $\Phi\subset \N$ satisfying
$$\sum_{n\in\Phi}\frac{\log_-|I_n|}{1+n^2}<\infty,$$
there exists a long sequence of intervals $\{J_n\}$ in $\R_+$ such that
$$\frac{\#(\Phi\cap J_n)}{|J_n|}\rightarrow 0$$
as $n\to\infty$.

\end{corollary}

\ms\no The first statement can be equivalently formulated as follows: if $L,\ti L\in \SS^2$ satisfy \eqref{eUP} and $q=\ti q$
on $(0,\e)$ for some $\e>0$ then $L\equiv \ti L$. The definition of long sequences of intervals (in the sense of Beurling and Malliavin)
was given in Section \ref{BM}.

\begin{proof}
Notice that if $\e_n=|I_n|$ then the condition in the second statement holds iff
\eqref{e2} in Theorem \ref{t0} is satisfied only for sequences $\Phi$ of interior density $0$.
Now the statement follows from Theorem \ref{t0} with $a=0$.
\end{proof}

%
%
%
%

\bs\bs\subsection{Three-interval statements}\label{s3int}

\ms\no Our goal in this section is to describe all possible counterexamples in the three interval problem. To formulate Theorem \ref{3int} below
we will need some preparation.

\ms\no Let $\mu$ be a positive measure on $\R$ such that $PW_a\subset L^2(\mu)$ for all $a>0$. It is well known, see for instance \cite{OS},
that any spectral measure of a \Sch operator from $\SS^2$ satisfies this condition. We define the type of $\mu$ as
$$T_\mu=\inf \{a|\ PW_a\textrm{ is dense in }L^2(\mu)\}.$$

\ms\no Let now $\L$ be a sequence satisfying the asymptotics \eqref{e000a}. Denote by $\eta$ the counting measure of $\L$.
It follows from the Type Theorem of \cite{Type, CBMS}, and in fact from much earlier results on the type problem, that $T_\eta=1$.
Moreover, as follows for instance from the results of  \cite{OS}, $L^2(\eta)=PW_1$ (as sets).

\ms\no Let now $0<c,d<1$ be real constants. In our next statement we will need  two even functions $f,g$,
$$f\in L^2(\eta)\ominus PW_{c},\ \  g\in L^2(\eta)\ominus PW_{d}$$
which satisfy certain asymptotics, see \eqref{e3}. As we will discuss after the statement, such functions form
dense subsets in the corresponding infinite-dimensional subspaces of $L^2(\eta)$. Recall that by $\Gamma(\L)$ we denote
the characteristic sequence of $\L$, or equivalently the sequence of pointmasses of the even operator, see Section \ref{even}.

\begin{theorem}\label{3int} Let $\L$ be a sequence satisfying the asymptotics \eqref{e000a}, $\Gamma(\L)=\{\gamma_n\}$. Denote by $\eta$ the counting measure of $\L$. Let
 $a,b\in [0,1/2)$ be arbitrary constants

\ms\no 1) for any two real even functions $f,g$,
$$f\in L^2(\eta)\ominus PW_{2a},\ \  g\in L^2(\eta)\ominus PW_{2b}$$
 satisfying
\begin{equation}|n|f(\lan)=a_n,\ |n|g(\lan)=a_n\left(1+\frac{b_n}{|n|+1}\right),\label{e3}\end{equation}
 for some $ a_n,b_n\in l^2$, $b_n\geq -|n|-1$,
the measures $\mu=\sum \alpha_n\delta_{\lan}$ and  $\ti\mu=\sum \ti\alpha_n\delta_{\lan}$, where
$$\alpha_n=\frac {f(\lan)}2 + \sqrt{\frac{f^2(\lan)}4 + \gamma^2_n\frac {g(\lan)}{f(\lan)}},$$
\begin{equation}\ti\alpha_n=-\frac {f(\lan)}2 + \sqrt{\frac{f^2(\lan)}4 + \gamma^2_n\frac {g(\lan)}{f(\lan)}}\label{e5}\end{equation}
when $a_n\neq 0$ and $\alpha_n=\ti\alpha_n=1+c_n/n$ for some $c_n\in l^2$ when $a_n=0$,
are spectral measures for two \Sch  operators $L,\ti L\in \SS^2$  such that
$$\sigma_{DD}=\ti\sigma_{DD}=\L,\ q=\ti q\textrm{ on }[0,a\pi]\cup [(1-b)\pi,\pi].$$

\ms\no 2) For any two spectral measures of Schr\"odinger operators $L,\ti L$ on $[0,\pi]$ such that
$$\sigma_{DD}=\ti\sigma_{DD}=\L,\ q=\ti q\textrm{ on }[0,a\pi]\cup [(1-b)\pi,\pi],$$
their pointmasses must satisfy
$$\alpha_n-\ti\alpha_n=f(\lan), \ \ti\alpha^{-1}_n-\alpha^{-1}_n=\gamma_n^{-2}g(\lan),$$
for some even functions $$f\in L^2(\eta)\ominus PW_{2a},\ g\in L^2(\eta)\ominus PW_{2b},$$
satisfying \eqref{e3} with some $ a_n,b_n\in l^2$, $b_n\geq -|n|-1$.

\end{theorem}

\begin{remark}\normalfont
 In regard to the functions $f$ and $g$ from the statement, note that since $2a, 2b <1$ and $L^2(\eta)=PW_1$, the orthogonal complements of $PW_{2a}$
 and $PW_{2b}$ in $L^2(\eta)$ are large infinite-dimensional subspaces.

\ms\no If a function $h$ belongs to $ L^2(\eta)\ominus PW_{2a}$, then $p(z)=\textrm{Re }h(z)-\textrm{Re }h(-z)$ is an odd real function from the same subspace.
The function $f(z)=p(z)/z$ will satisfy  the conditions for $f$. Moreover, the set of such functions is dense in the space of
all even  functions from $L^2(\eta)\ominus PW_{2a}$.

\ms\no To choose $g$, assume for instance that $a\geq b$. Choose two more real odd functions $p_1, p_2$ in $ L^2(\eta)\ominus PW_{2a}$ as described above. Then the function
$r(z)=p_1(z)p_2(0)-p_1(0)p_2(z)$ will have a double zero at 0. The function $ g=f+r/z^2$ will satisfy the conditions for $g$. If $a< b$ one
may proceed in the opposite direction, first choosing $g$ and then $f$.

\end{remark}

\begin{proof}[Proof of Theorem \ref{3int}] 1)
If $\alpha_n$ and $\ti\alpha_n$ satisfy \eqref{e5} then $\mu_-=\mu$ and $\ti\mu_-=\ti\mu$ satisfy
the asymptotics \eqref{asymp2}. Hence they are the spectral measures of some \Sch  operators $L$ and $\ti L$. Since $\mu_--\ti\mu_-=f\eta\perp PW_{2a}$,
$q=\ti q $ on $(0,a\pi)$ by Lemma \ref{l5}. By Lemma \ref{leven}, $\mu_+-\ti\mu_+=g\eta$ and therefore by Lemma \ref{l5}, $q=\ti q $ on $((1-b)\pi,\pi)$.

\ms\no 2) In the opposite direction, if $L$ and $\ti L$ are as in the statement, then by Lemma \ref{l5} the measures $\mu_--\ti\mu_-$ and
$\mu_+-\ti\mu_+$ have spectral gaps. Hence they annihilate the corresponding $PW$-spaces, i.e.
 $\mu_--\ti\mu_-=f\eta$ and
$\mu_+-\ti\mu_+=g\eta$ for some even functions $f\in L^2(\eta)\ominus PW_{2a},\ \  g\in L^2(\eta)\ominus PW_{2b}$. Lemma \ref{leven} and
the asymptotics \eqref{asymp1}, \eqref{asymp2} imply \eqref{e3}.
\end{proof}

\ms\no Theorem \ref{3int} describes all possible counterexamples for the three-interval problem. As we can see, this set is richer  than previously thought, with the original counterexample from \cite{SG2} discussed in the introduction corresponding to the case $f=g$, i.e., $b_n\equiv 0$ in the notations of the statement. On the other hand, all of the counterexamples must be close to the original in the sense that the difference between
$f$ and $g$ has higher order of decay in comparison with either function.

\bs\subsection{Uniqueness in the 3-interval problem near the even operator}\label{3intU}

\ms\no In this section we will look at the 3-interval problem from a slightly different point of view. We will ask if there exist
operators $L$ uniquely determined by the '3-interval' information. Our result here can be viewed as a set-up for the problem
of description of such operators discussed in the next section.

\ms\no Let $\mu=\mu_-=\sum \alpha_n\delta_{\lan}$ denote the spectral  measure of an  operator $L$ from $\SS^2$ as defined in Section \ref{functions}.
For $\L\subset \Z$ denote
$$D'_*(\L)=\max\{D_*(\Phi)|\ \Phi\subset\L, \{n,n+1\}\not\subset \Phi, \forall n\in \Z\}.$$
Once again, for a discrete sequence $\L\subset\R$ we denote by $\Gamma(\L)$ its characteristic sequence
from Section \ref{even}.

\begin{theorem}\label{tUoper} Let $ 0\leq a<1/2$ and let $\{\e_n\}_{n\in\N}$ be a sequence of positive numbers such that for
$\Phi\subset\N$ the inequality
\begin{equation}\sum_{n\in\Phi}\frac {\log_-\e_n}{1+n^2}<\infty\label{e22}\end{equation}
implies
$$D'_*(\Phi)\leq a.$$
Then
any operator $L\in \SS^2$, satisfying
\begin{equation}|\alpha_n-\gamma_n|<\e_n,\label{e21}\end{equation}
where $\{\gamma_n\}=\Gamma(\sigma_{DD})$, is uniquely determined by its spectrum $\sigma_{DD}$ and its potential
$q$ on $(0,d\pi)\cup ((1-d)\pi,\pi)$ for any $d>a$. I.e., if $L$ satisfies \eqref{e21} then any $\ti L\in \SS^2$ with $\ti\sigma_{DD}=\sigma_{DD}$ and
$\ti q=q$ on $(0,d\pi)\cup ((1-d)\pi,\pi)$ for some $d>a$ must be identical to $L$.

\end{theorem}

\begin{remark}\normalfont The last statement displays an interesting phenomenon. As we can see from the original counterexample, if an operator is close to even (but not even)
in terms of potential (in the sense of direct problem), namely if its potential is even on $[0,d]\cup [1-d,\pi]$, then it is not uniquely defined by the corresponding
mixed spectral data. However, if an operator is close to even in terms of the spectral measure (in the sense of inverse problem) then it is uniquely defined by the mixed spectral data. This property is yet to be fully understood. One of its particular consequences is a statement typical for the
area of the Uncertainty Principle: an operator cannot be close to even (without being even) simultaneously with respect to its potential and spectrum, in the sense of the last statement.

\ms\no It is clear from the proof below that instead of $\ti\sigma_{DD}=\sigma_{DD}$ one could require that the intersection of the spectra have large enough density.
\end{remark}
\begin{proof} Let $\L=\sigma_{DD}$.
Let $\ti L$ be an operator as in part 1) of the statement, different from $L$. Then $f(\lan)=\alpha_n-\ti\alpha_n$ and $g(\lan)=\beta_n-\ti\beta_n$ are the functions
like in the statement of Theorem \ref{3int}. Notice that
$$g=\frac{\gan^2 f}{(\alpha_n +f)\alpha_n}$$
and
$$f-g=\frac{(\alpha_n^2-\gan^2) f +
\alpha_n f^2}{(\alpha_n +f)\alpha_n}.$$
Since the measure $(f-g)\eta$ annihilates $PW_{2d}$, by Corollary \ref{l2},
$$\sum_{n\in\Delta}\frac{\log_-|(\alpha_n^2-\gan^2) f -\alpha_n f^2|}{1+n^2}<\infty$$
for some $\Delta,\ \pi D_*(\Delta)\geq 2d$. Let $\Phi,\ \pi D_*(\Phi)= a $, be a
sequence of maximal density satisfying \eqref{e22}. Notice that at each point of  $\Z\setminus \Phi$,
either $(\alpha_n^2-\gan^2) f -\alpha_n f^2>0$ or $|f|\lesssim |\alpha_n-\gan|$.
Denote the  set of indices corresponding to the first case by $\Upsilon_1$ and those for the second case  by $\Upsilon_2$.
Then $D_*(\Upsilon_2\cap \Delta)=0$. Hence, up to a sequence of density 0,
$\Delta$ must be contained in $\Phi\cup\Upsilon_1$.

\ms\no But the measure $(f-g)\eta$ is positive on $\{\lan\}_{n\in\Upsilon_1}$. Hence, its negative part
must be contained in $\Phi$, up to a sequence of density 0. By Corollary \ref{forTeven}, the support of the
negative part must have density at least $d$, which contradicts the relation $d>a$.

\end{proof}

\section{Further Examples and Open Problems}

\bs\subsection{Problem 1: Description of defining sequences in the case of condition-free endpoint in the two-interval problem.}\label{P1}

In Horvath' theorem discussed in Section \ref{sHorvath} the operator is recovered from a part of its potential on $(0,a\pi)$ and
the values of the Weyl function $m_-$, obtained by fixing a boundary condition at the opposite point $\pi$. What if
we consider a similar problem with a condition-free end at $\pi$, i.e., try to recover an operator from its potential near $0$ and
the function $m_+$ obtained by fixing a boundary condition at the same endpoint $0$? In this case the statement
analogous to Horvath' theorem fails.

\begin{example}
The simplest counterexample is a modification of the 3-interval counterexample from \cite{SG2} mentioned before.
Let the operators $L$ and $\ti L$  from $\SS^2$ be such that
$q=\ti q \textrm{ on }(0,(1-\e)\pi)$ and $q(x)=\ti q((2-\e)\pi-x)$ on $((1-\e)\pi,\pi)$. Consider an auxiliary operator
$L_a$ on $(0,(2-\e)\pi)$ whose potential is defined as $q_a=q$ on $(0,\pi)$ and $q_a(x)=q((2-\e)\pi)-x)$ on $(\pi,(2-\e)\pi)$.
Let $\ti L_a$ be an operator such that $\ti q_a=\ti q$ on $(0,\pi)$ and $\ti q_a(x)=q((2-\e)\pi-x)$ on $(\pi,(2-\e)\pi)$.
Note that then the DD spectra of $L_a$ and $\ti L_a$ coinside.
Denote this sequence by $\L$. It is not difficult to see the Weyl functions $m_+$ and $\ti m_+$ of the operators $L$ and $\ti L$
will coincide on $\L$. It is left to notice that the density of $\L$ is $2-\e$, i.e., much larger than $\e$.
\end{example}

\ms\no Although, as we can see from the last example, not any sequence $\L$ of large enough density can be used in our modified two-interval problem, there are
some sequences which can  be used, as shown by the following.

\begin{proposition}
Let $g\in L^2((0,\pi-\e)),\ \e<\pi/2$ and let $\Sigma=\sigma_{DD}$ be the spectrum of  the \Sch operator from $\SS^2((0,\pi-\e))$ with $q=g$.
Then any two $L,\ti L\in \SS^2((0,\pi))$, such that $q=\ti q=g$ on $(0,\pi-\e)$ and $m_+=\ti m_+$ on
a subsequence $\L$ of $\Sigma$, $\pi D^*(\L)>2\e$, must be identical.
\end{proposition}

\begin{proof} Let $u_z$ and $\ti u_z$ denote the solutions satisfying the Dirichlet boundary conditions at $0$. For $z\in\Sigma$ the restrictions
	of $u_z$ and $\ti u_z$ on $(\pi-\e,\pi)$ satisfy Dirichlet boundary conditions at $\pi-\e$. Therefore the $m$-functions
	$m^\e_+,\ \ti m^\e_+$ corresponding to the restrictions of $L$ and $\ti L$ on $(\pi-\e,\pi)$ coincide with $m_+,\ti m_+$ on $\Sigma$.
Since $m_+=\ti m_+$ on $\L$, we have that $m^\e_+=\ti m^\e_+$ on $\Lambda$. The difference
$m^\e_+-\ti m^\e_+$ is a Herglotz integral with poles on a sequence of density at most $2\e$ and zeros on a sequence of larger density.
By a version of the Beurling-Malliavin theorem it implies that the difference is identically zero
and $q\equiv \ti q$.

\end{proof}

\ms\no Let $g\in L^2(0,\pi-\e)$. Denote by $\SS^2_g$ the set of all operators from $\SS^2$ such that $q=g$ on $(0,\pi-\e)$.
The natural problem which arises from the last two propositions is to describe  the set of sequences $\L$ such that
the values of the Weyl function $m_+$ on $\L$ determine an operator uniquely within the class $\SS^2_g$. The example
 shows that the set does not contain all of the sequences of proper density and the last statement
shows that the set is non-empty.

\bs\subsection{Problem 2: Description of unique operators in the three-interval case.}\label{P2}

\ms\no Let $0\leq a<b\leq 1$ and $a+(1-b)>1/2$.
As follows from the discussion in the previous section, there exist \Sch operators  $L\in \SS^2$ on $(0,\pi)$ with the following
uniqueness property: If $\ti L\in \SS^2$ is another \Sch operator such that $\ti q=q$ on $(0,a\pi)\cup (b\pi,\pi)$
and $\pi D^*(\ti\sigma_{DD}\cap \sigma_{DD})>2(b-a)$ then $\ti L \equiv L$.

\ms\no This rises a natural question of description of all such 'unique' operators, i.e., operators
uniquely determined by the restriction of their potentials on $(0,a\pi)\cup (b\pi,\pi)$ and
a properly sized subsequence of the spectrum. The example provided by Theorem \ref{tUoper} presents an operator
which is close to even in the spectral case. The question one may start with is whether there are other examples.

\ms\no While our discussion in this paper concentrates on the $\SS^2$ case, it can be extended to regular case without much difficulty. Further cases
of this problem may concern similar examples and descriptions in non-regular cases. If $q$ is unsummable, the correct question would be to
describe $L$ such that any $\ti L$ with the above properties, and such that $q-\ti q$ is small (summable, for instance), must coincide with $L$.

\ms\no The three-interval case is only a model case in which we already see some of the difficulties that were not present in the two-interval problem.
The ultimate goal in this and similar problems would be to consider other subsets of the interval. We continue this discussion in the last subsection.

\bs\subsection{Problem 3: Uncertainty in other types of spectral data}\label{P3a}

\ms\no The version of the two-spectra  problem treated in Section \ref{2int} is not the only case when a problem
 of Uncertainty Quantification appears naturally in spectral settings.  Let us give another example of such a problem.

\ms\no Let $L\in \SS^2$ be a \Sch operator with the spectral measure $\mu_-=\sum \alpha_n\delta_{\lan}$. Let
$\EE=\{\e_n\}$ be a sequence of positive numbers. Denote by $\EE_L$ the set of operators $\ti L\in \SS^2$
such that $\ti\mu_- =\sum\ti \alpha_n\delta_{\lan}$, $|\alpha_n-\ti \alpha_n|\leq\e_n$. Similarly
to Section \ref{2int}, denote by $U(\EE_L)$ the infimum of $a$ such that the values of the potential on $(0,a\pi)$
determine a \Sch operator from $\EE_L$ uniquely.

\begin{proposition}
$$U(\EE_L)=\pi \sup \left\{ D_*(\Phi)\ :\ \sum_{n\in\Phi}\frac{\log_-\e_n}{1+n^2}<\infty\right\}.$$

\end{proposition}

\ms\no The proof follows easily from Lemma \ref{l5} and Corollary \ref{l2}.

\ms\no The estimates of the size of uncertainty in other variations of spectral problems may require different techniques.
The definition of $U(D)$ for the spectral data $D$ which worked for us in the problems considered in this paper may need to be further developed for other kinds of data. Let us point out, for instance, that the present
definition of $U$  will not work properly if one replaces $\mu_-$ with $\mu_+$ in the definition of $\EE_L$ above. Even though
an obvious modification of $U$ by replacing  $(0,a\pi)$ with $((1-a)\pi,\pi)$ will fix this particular problem,
one would like to have a more universal definition.

\bs\subsection{Problem 4: Further connections between mixed spectral problems and completeness problems.}\label{P3}

\ms\no Horvath' Theorem \ref{Horvath} has established a connection between the Beurling-Malliavin (BM) problem on completeness of systems of exponentials
in $L^2$ on an interval and the original case of the mixed spectral problem, the two-interval case without a condition-free endpoint.
Even though cases of multiple intervals were considered in the literature (see for instance
\cite{Tru}, Chapter 4, Problem 10a) similar connections are yet to be found.
It is interesting to observe that the analog of the BM theorem with one interval replaced with any other subset of the line, including
the next simplest case of a union of two intervals, does not exist. Despite a number of deep results on completeness of exponential systems in $L^2$ over general sets, see for
instance \cite{Olevsky}, there is no formula for the radius of completeness or  a good idea on what could replace such a formula, even in the case of two intervals.

\ms\no As we saw in our discussion in Section \ref{s3int}, many of the same complications appear in mixed spectral problems when moving from the two-interval to the three-interval and multiple-interval case.
Moreover, analogies between the Weyl transform and the Fourier transform, together with the use of the latter in BM theory,
suggest that the mixed spectral problems for more general subsets of the interval must be closely related to BM problems over general sets.
Finding such connections, i.e., formulating an analog of Horvath' theorem for more general subsets seems to be an interesting and challenging problem.


\begin{thebibliography}{24}


\bibitem{BBP}{\sc Baranov, A., Belov, Yu. and Poltoratski, A.} {\it De Branges functions for Schroedinger equations,} Preprint, 2016


\bibitem{Beurling1}{\sc Beurling, A.} {\it On quasianalyticity and
general distributions,} Mimeographed lecture notes, Summer
institute, Stanford University (1961)

\bibitem{BM1}  {\sc Beurling, A., Malliavin, P.} {\it
On Fourier transforms of measures with compact support,}
Acta Math.  107 (1962), 291--302

\bibitem{BM2}  {\sc Beurling, A., Malliavin, P.} {\it
On the closure of characters and the zeros of entire functions,}
Acta Math.  118 (1967), 79-93

\bibitem{Krein}{\sc Birman, M. Sh., Yafaev, D. R.}{\it : The spectral shift function. The papers of M. G.
Krein and their further development.}  St. Petersburg Math. J. 4, 833-870 (1993).

\bibitem{Borg} {\sc Borg, G.} {\it Eine Umkehrung der Sturm-Liouvilleschen Eigenwertaufgabe,} Acta Math.,
79 (1946), 1-96


\bibitem{dBr}  {\sc De Branges, L.} {\it Hilbert spaces of entire functions.} Prentice-Hall,
Englewood Cliffs, NJ, 1968


\bibitem{D} {\sc Donoghue, W. F., Jr.} {\it On the perturbation of spectra.} Comm. Pure Appl. Math. 18 1965 559-–579


\bibitem{Dym}{\sc Dym, H. }{\it An introduction to de~Branges spaces of entire functions with applications to differential equations of the Sturm-Lioville type,} Advances in Math. 5 (1971), 395-471


\bibitem{DM}{\sc Dym H, McKean H.P.} {\it Gaussian processes, function theory and the inverse spectral problem}
Academic Press, New York,  1976




\bibitem{DGS}{\sc Rafael del Rio, Fritz Gesztesy, and Barry Simon,}{\it Inverse spectral analysis with partial information on the potential. III. Updating boundary conditions, } Internat. Math. Res. Notices, 15 (1997), 751--758


    \bibitem{GL} {\sc  Gelfand, I, M.,  Levitan, B. M.} {\it On the determination of a differential equation from its
spectral function} (Russian), Izvestiya Akad. Nauk SSSR, Ser. Mat., 15 (1951), 309--360;
English translation in Amer. Math. Soc. Translation, Ser. 2, 1 (1955), 253--304.


\bibitem{SG1} {\sc Gesztezy F., Simon B.}\textit{ A new approach to inverse spectral theory, II. General real potentials and the connection to the spectral measure,} Ann. of Math. 152 (2000), 593--643


\bibitem{SG2} {\sc Gesztezy F., Simon B.}\textit{ Inverse spectral analysis with partial information on the potential, II. The case of discrete spectrum,} Trans. AMS 352 (2000), 2765--2787

\bibitem{SG3} {\sc Gesztesy, F., Simon, B.} {\it Uniqueness theorems in inverse spectral theory for one-dimensional Schr\"odinger operators.}
Trans. Amer. Math. Soc. 348 (1996), no. 1, 349–-373




\bibitem{GT} {\sc Gesztesy, F., Tsekanovskii, E.} {\it On matrix-valued Herglotz functions.} Math. Nachr. 218 (2000), 61–-138.



\bibitem{HL}  {\sc Hochstadt,  H., Lieberman, B.} {\it
An inverse Sturm-Liouville problem with mixed given data,}
SIAM J. Appl. Math.  34  (1978),  676--680


\bibitem{Horvath} {\sc Horv\'ath M.} {\it Inverse spectral problems and closed exponential systems,} Ann. Math.
162 (2005), 2, 885--918.

\bibitem{HNP}{\sc Hruschev S.,  Nikolskii, N.,  Pavlov,  B.} {\it
Unconditional bases of
exponentials and of reproducing kernels,}
  Lecture Notes in Math., Vol.  864,  214--335

\bibitem{Koosis} {\sc  Koosis, P.} {\it The logarithmic integral, Vol. I \ \& II,} Cambridge Univ. Press, Cambridge, 1988


\bibitem{LS}{\sc Levitan, B.M., Sargsjan, I.S. } {\em Sturm-Liouville and Dirac operators.} Kluwer, Dordrecht, 1991

\bibitem{Levinson}{\sc Levinson, N.} {\it Gap and density theorems,} AMS
Colloquium Publications, 26 (1940)



\bibitem{Lub}
{\sc D. S. Lubinsky}  {\it A Survey of Weighted Polynomial Approximation with Exponential
Weights,} Surveys in Approximation Theory, 3, 1�105 (2007)


\bibitem{MIF1} {\sc  Makarov, N.,  Poltoratski, A.} {\it Meromorphic inner functions, Toeplitz kernels, and the uncertainty principle,} in {\it Perspectives in Analysis}, Springer Verlag, Berlin, 2005, 185--252

\bibitem{MIF2} {\sc  Makarov, N.,  Poltoratski, A.} {\it  Beurling-Malliavin theory for Toeplitz kernels,}
Invent. Math., Vol. 180, Issue 3 (2010), 443--480

\bibitem{Etudes} {\sc  Makarov, N.,  Poltoratski, A.} {\it Etudes for the inverse spectral problem,}
preprint

\bibitem{MPS}{\sc Makarov, N., Poltoratski, A., Sodin, M.} {\it Lectures on Linear Complex Analysis}, Book in preparation.


\bibitem{M1} {\sc  Marchenko, V.} {\it
Some questions in the theory of one-dimensional linear differential operators of the second order, I,} Trudy Mosk. Mat. Obsch. 1 (1952), 327--420.


\bibitem{M2}{\sc  Marchenko, V.} {\it Sturm-Liouville operators and applications.}  Birkhauser, Basel, 1986

 \bibitem{Meg}
{\sc Mergelyan, S.}  {\it  Weighted approximation by polynomials,} Uspekhi mat. nauk, 11 (1956), 107--152,
English translation in Amer. Math. Soc. Translations, Ser 2, 10 (1958), 59--106

\bibitem{Rem}
C. Remling, \textit{Schr\"odinger
operators and de Branges spaces}, J. Funct. Anal. 196 (2002), 323--394.


\bibitem{Polya} {\sc Mitkovski, M. and Poltoratski, A.} {\it Polya sequences, Toeplitz kernels and gap theorems,}
Advances in Math., 224 (2010), pp. 1057--1070

\bibitem{Det} {\sc Mitkovski, M. and Poltoratski, A.} {\it On the determinacy problem for measures,} Invent. Math., 202 (2015), 1241--1267



\bibitem{Olevsky}{\sc Olevskii, A. and Ulanovskii, A.} {\it Functions with disconnected spectrum: sampling, interpolation, translates.}
Univ. Lecture Series, AMS 2016


\bibitem{OS}  {\sc  Ortega-Cedr\'a, J., Seip, K.} {\it Fourier frames,} Annals of Math. 155 (2002), 789--806


\bibitem{Pa}  {\sc  Pavlov, B.} {\it The basis property of a system of exponentials and the condition of Muckenhoupt,} Dokl. Acad. Nauk SSSR 247, (1979), 37--40




\bibitem{Poly} {\sc A. Poltoratski,} {\it Bernstein's problem on weighted polynomial approximation,}
Operator-Related Function Theory and Time-Frequency Analysis, Abel Symposia, vol 9., Springer 2015, pp 147--171.


\bibitem{PoltKrein}
{\sc Poltoratski, A.} {\it Kre\u\i n's spectral
shift and perturbations of spectra of
rank one,}  Algebra i
Analiz, {\bf 10} (1998), no.~5, 143--183; translation
in {\em St.\
Petersburg Math.\ J.} {\bf 10} (1999), no.~5, 833--859






\bibitem{Gap} {\sc Poltoratski, A.} {\it Spectral gaps for sets and measures}, Acta Math., 2012, Volume 208, Number 1, pp. 151-209.


\bibitem{Type}{\sc
   Poltoratski, A.},
   {\it A problem on completeness of exponentials},
    Ann. Math.,
     178,
   983--1016,
   2013


\bibitem{CBMS} {\sc Poltoratski, A.}  {\it  Toeplitz Approach to Problems of the Uncertainty Principle,} CBMS series, Providence, Rhode Island: Published by the American Mathematical Society, 2015

    \bibitem{Tru}
J. Poshel, E. Trubowitz, \textit{Inverse Spectral Theory},
Academic Press, New-York, 1987.


    \bibitem{Remling}{\sc Remling, C.}{\it Schr\"odinger operators and de Branges spaces}, J. Func. Anal., Volume 196, Issue 2, 2002, 323�394


    \bibitem{Rom}{\sc Romanov, R.}{\it Canonical Systems and de Branges Spaces}, Preprint (Lecture Notes), 2016


\bibitem{S}{\sc Simon, B.} {\it A new approach to inverse spectral theory, I. Fundamental formalism,} Ann. of Math. 150 (1999), 1029-1057

\end{thebibliography}
\end{document}